\documentclass[12pt]{amsart}

\usepackage{amsmath,amssymb,amsthm}

\usepackage{amsmath, amssymb}
\usepackage{amsthm, amsfonts, mathrsfs}
\usepackage{fullpage}
\usepackage{amsfonts,graphicx}

\theoremstyle{plain}
\newtheorem{Theo}{Theorem}[section]
\newtheorem{lem}[Theo]{Lemma}

\newtheorem{prop}[Theo]{Proposition}

\theoremstyle{plain}
\theoremstyle{definition}

\theoremstyle{remark}
\newtheorem{Rema}[Theo]{Remark}
\newtheorem*{rema*}{Remark}

\newcommand{\NN}{\mathbb{N}}
\newcommand{\eps}{\varepsilon}

\newcommand{\RR}{\mathbb{R}}

\author[T. Hmidi]{Taoufik Hmidi}
\address{IRMAR, Universit\'e de Rennes 1\\ Campus de
Beaulieu\\ 35~042 Rennes cedex\\ France}
\email{thmidi@univ-rennes1.fr}
\author[F. Rousset]{Fr\'ed\'eric Rousset}
\address{IRMAR, Universit\'e de Rennes 1\\ Campus de
Beaulieu\\ 35~042 Rennes cedex\\ France}
\email{frederic.rousset@univ-rennes1.fr}

\date{}

\begin{document}

\title[Axisymmetric Boussinesq system]
{Global well-posedness for the Navier-Stokes-Boussinesq system with axisymmetric data}
\maketitle
\begin{abstract}
In this paper we prove  the global well-posedness   for a  three-dimensional  Boussinesq system   with axisymmetric initial data. This system couples  the Navier-Stokes equation with a transport-diffusion  equation governing the temperature. Our result holds uniformly with respect to the heat conductivity 
 coefficient $\kappa \geq 0$ which may vanish. 
\end{abstract}
\section{Introduction}
 The   Boussinesq system  is  widely used to model the dynamics of the  ocean or  the atmosphere.
    It  arises from the  
    density dependent  incompressible  Navier-Stokes equations by using the so-called
    Boussinesq approximation which consists in neglecting the density dependence in all the terms
     but  the one involving the gravity. 
      This system writes  
\begin{equation}
    \label{bsintro}
\left\{ \begin{array}{ll}
\partial_{t}v+v\cdot\nabla v  -  \Delta v+\nabla p=\rho e_z,\quad (t,x)\in \RR_+\times\RR^3,\\
\partial_{t}\rho+v\cdot\nabla\rho= \kappa \Delta \rho,\\
\textnormal{div}\,  v=0,\\
v_{| t=0}=v_{0},   \quad \rho_{| t=0}=\rho_{0}.
\end{array} \right.
\end{equation}
Here, the velocity  $v=(v^1,v^2,v^3)$ is a three-component vector field with zero divergence,  the scalar
function $\rho$ denotes the density or the temperature and $p$ the pressure of the fluid. 
 The  coefficient  $\kappa \geq 0$  is a  Reynolds number which takes into account
   the  strength of heat conductivity. Note that we have assumed that the viscosity coefficient
     is one, one can always reduce the problem to this situation by a change of scale
      (as soon as the fluid is assumed to be viscous)  
      which is not important for global well-posedness issues with data of arbitrary size
       that we shall consider.
  The  term  $\rho e_{z}$  where $e_{z}= (0, 0, 1)^t$ takes into account the influence of
   the gravity  and the stratification  on the motion of the fluid. 
Note 
 that  when the initial density $\rho_0$ is identically zero (or constant)  then the above system  reduces  to the
classical  incompressible Navier-Stokes equation:
\begin{equation}
      \label{ns}
\left\{ \begin{array}{ll}
\partial_{t}v+v\cdot\nabla v-\Delta v+\nabla p=0\\
\textnormal{div}\,  v=0\\
v_{| t=0}=v_{0}.
\end{array} \right.
\end{equation}
 From this observation, one cannot expect to have a better theory for the Boussinesq system
  than for the Navier-Stokes equations.
 The existence of global weak solutions in the energy space for \eqref{ns}
 goes back to J. Leray \cite{Leray}. However the uniqueness of these solutions
 is only known in space dimension two. It is also well-known that  smooth solutions are global
 in dimension two and  for higher dimensions when the data are small in some critical spaces; see for instance \cite{Lemar} for more detailed  discussions.
  In a similar way, the global well-posedness  for two-dimensional Boussinesq systems
   which has recently drawn a lot of attention seems   to be in a satisfactory state. More precisely 
 global  well-posedness   has been shown  in  various  function spaces  and for
different viscosities, we refer  for example to \cite{ah,Brenier,
ha,dp1,dp,dp2,hk1,hk,HKR1,HKR2}.
  For  three-dimensional systems  few results are known about global existence. We  can quote 
the result of R. Danchin and M. Paicu \cite{dp} who proved a global well-posedness result for
small initial data belonging  to some critical Lorentz spaces.

 Let us recall 
 that 
  it is still not known  if smooth solutions with large initial data for the Navier-Stokes equations
   can blow-up in finite time  in dimension $3$.  Only  some
  partial results are known.   For example,  in a recent series of papers
  \cite{CG,CG1,CG2}
     global existence in dimension
 three is established   for   initial data which are not small in any critical space but which have some
  special structure (oscillations or  slow variations in one direction).
   Another     interesting   case  of    global existence for \eqref{ns}  corresponding  to large initial data but with special structure is  the more classical case of  axisymmetric solutions without swirl.
   Our aim in this paper is to establish the corresponding global well-posedness result for the
    three-dimensional Boussinesq system.
   
Before  stating our main result,   let us   describe the classical result for the  Navier-Stokes  equation.
  It is well-known that the  control of the vorticity  $\omega$ which  is  the vector defined by
 $\omega={\rm curl}\,  v$ and  solving the vorticity  equation
$$
\partial_t \omega+v\cdot\nabla\omega-\Delta\omega=\omega \cdot\nabla v
$$
is crucial in order to get global well-posedness results.
According to the classical  Beale-Kato-Majda  criterion \cite{bkm} 
 the control of the vorticity in $L^1_{loc}(\mathbb{R}_{+}, L^\infty)$ is sufficient  to get   the  global existence
  of smooth solutions.
  The main difficulty arising in dimension  three is the lack
of information about the influence of  the vortex-stretching term
$\omega \cdot\nabla v$  on  the  motion of the fluid.
Let us now consider   a   vector field $v$  which is  axisymmetric
 without swirl,   this means that   it has
the form:
\begin{equation}
\label{vform}
v(t,x) = v^r(t,r, z)e_r + v^z (t,r, z)e_z,\quad x=(x_{1},x_{2},z),\quad r=({x_{1}^2+x_{2}^2})^{\frac12},
\end{equation}
where  $\big(e_r, e_{\theta} , e_z\big)$ is the local  orthonormal basis of $ \mathbb{R}^3$  corresponding to cylindrical 
 coordinates. Note that we assume that  the velocity is  invariant by rotation around the  vertical axis
(axisymmetric flow) and that the component $v^{\theta}$ of $v$ about $e_{\theta}$ identically vanishes
 (without swirl).   For these   flows, 
 we have:  
$$
\omega=(\partial_zv^r-\partial_rv^z) e_{\theta}:=\omega_\theta e_\theta, \quad
  \omega \cdot \nabla v= {v^r \over r }\omega.$$
 In particular   $\omega_\theta$ satisfies  the equation
\begin{equation}
 \label{tourbillon}
\partial_t \omega_\theta +v\cdot\nabla\omega_\theta-\Delta\omega_\theta
+\frac{\omega_\theta}{r^2} =\frac{v^r}{r}\omega_\theta.
\end{equation}
 The crucial fact is then that  the quantity $\zeta:=\frac{\omega_\theta}{r}$  evolves according  to the equation
\begin{equation*}
\partial_t \zeta+v\cdot\nabla \zeta- \big( \Delta +\frac2r\partial_r\big) \zeta=0
\end{equation*}
from which we  get 
 that for all $p\in[1,\infty]$
$$
\|\zeta(t)\|_{L^p}\le\| \zeta_0\|_{L^p}.
$$
It was shown by M. Ukhoviskii and V. Yudovich \cite{Ukhovskii}  and independently by O. A. Ladyzhenskaya \cite{LA}  that these  new a priori estimates  are  strong enough to prevent the formation of singularities in
finite time for axisymmetric flows without swirl:  the system \eqref{ns} has a unique global solution for  $v_0\in H^1$ such that
$\omega_0,\frac{\omega_0}{r}\in L^2\cap L^\infty.$  We point out that  the result is uniform with respect to vanishing viscosity and thus  there is no blowup even for the Euler equation.
Note  that in term of Sobolev
regularities these assumptions are satisfied when $v_0\in H^s$ with $s>\frac72.$
 This regularity assumption is not optimal and  has been weakened in   
      \cite{lmnp} for   $v_0\in H^2$ and   more recently  in \cite{a} for  $v_0\in H^{\frac{1}{2}}$.
      
 Our aim here is to  extend  these classical results  for the Navier-Stokes equation
   to  Boussinesq systems. The equation for $\zeta= \omega_{\theta}/r$
   becomes     
\begin{equation}
\label{mBouss}\partial_t \zeta+v\cdot\nabla \zeta-  \big( \Delta\Gamma+\frac2r\partial_r\big) \zeta= - {\partial_{r} \rho \over r }
\end{equation}
and thus the difficulty is to use some a priori estimates on $\rho$ to control the   term in the right-hand side
 of  \eqref{mBouss}. The rough idea is that  on the axis $r=0$ the singularity ${1\over r}$
 scales  as a derivative and hence that  the forcing term $\partial_{r} \rho/r$ can be thought as a Laplacian
  of $\rho$ and thus one may try to use  smoothing effects to control it.   
  Nevetheless,   when $\kappa=0$,
    since there is no 
     smoothing effects on $\rho$, one can hope to 
      compensate the loss of derivatives in the right hand-side of \eqref{mBouss} only by 
       using the full smoothing effects of the heat type equation in the  left hand side.
       Note that this kind of estimates does not follow from energy estimates and hence
        the problem that one has to face is that the convection  term  that has to be handled in the process
         is not negligible: this approach naturally leads to some  restriction on the size of the data.
         
         In \cite{A-H-K0}, a global existence result for the system  \eqref{bsintro}, with $\kappa=0,$ was established but under some restrictive conditions on the support of the initial density
          namely that it does not intersect the  axis $r=0$. More precisely, 
     \begin{Theo}
 Let $v_0\in H^1$ be an axisymmetric divergence free  vector field without swirl  and such that $\frac{{\omega_0}}{r}\in L^2.$
 Let $\rho_0\in L^2\cap L^\infty$  axisymmetric and such that $\hbox{supp }\rho_0$
 does not intersect the axis $(Oz)$ and $\Pi_z(\hbox{supp }\rho_0)$ is a compact set.
 Then the system \eqref{bsintro}, with $\kappa=0,$ has a unique global solution $(v,\rho)$ such that
 $$
 v\in{C}(\RR_+;H^1)\cap L^1_{\textnormal{loc }}(\RR_+;W^{1,\infty}),$$
 $$ \frac{\omega}{r}\in L^\infty_{\textnormal{loc}}(\RR_+;L^2),\quad
 \rho\in L^\infty_{\textnormal{loc}}(\RR_+;L^2\cap L^\infty).
 $$
 Here $\Pi_z$ denotes the orthogonal projector over $(Oz).$
  \end{Theo}
    Since  to bound the quantity $\|\omega/r\|_{L^\infty_tL^2\cap L^2_t\dot{H}^1}$ one needs   to estimate $\|\rho/r(t)\|_{L^2}.$ 
     The idea of the proof was to get an estimate from below for the distance of the support of $\rho$
      to  the  vertical axis (note that this distance remains  positive as long as the solution
        remains smooth because of  the assumption
       on the initial density).
    Note that   for this approach, it is crucial to have a transport equation   for  the density, it fails for  $\kappa>0$ because  $\rho$ cannot  be  supported away from the axis even if the initial data is.

     In this paper,    by using a different approach, which uses more deeply the structure
      of the coupling between the two equations of \eqref{bsintro},   we  remove the  assumption on the support of the density and we give a global well-posedness result with uniform bounds   with respect to the   heat conductivity $\kappa.$   Our main result reads as follows.

 \begin{Theo}\label{thm1}
 Consider the Boussinesq system \eqref{bsintro}  for  $\kappa\geq 0$.
 Let $v_0\in H^1$ be an axisymmetric divergence free  vector field without swirl  
  such that $\frac{{\omega_0}}{r}\in L^2$
 and let $\rho_0\in L^{2}\cap B_{3,1}^0$  an axisymmetric function. 
 Then there is  a unique global  solution $(v,\rho)$ such that
 $$
 v\in\mathcal{C}(\RR_+;H^1)\cap L^2_{\textnormal{loc}}(\RR_+; H^2)\cap L^1_{\textnormal{loc }}(\RR_+;B_{\infty,1}^{1}),$$
 $$ \frac{\omega}{r}\in L^\infty_{\textnormal{loc}}(\RR_+;L^2),\quad
 \rho\in \mathcal{C}(\RR_+;L^{2}\cap B_{3,1}^0).
 $$
 Moreover the  estimates in the above spaces   are  uniform  for  $\kappa \geq 0 $ in any 
  bounded set of $\mathbb{R}_{+}$.
  \end{Theo}
The definition of the Besov spaces $B_{p,q}^s$ is recalled below.

Let us give a few comments about our  result.

  \begin{Rema}
We find  at most an exponential growth for  the velocity:   there exists $C_{0}>0$
 depending only on the data such that  for every $\kappa \in [0,1]$, we have 
$$
\|v\|_{L^\infty_tH^1\cap L^2_t H^2}\le C_0 e^{C_0 t}.
$$ 
  \end{Rema}
  \begin{Rema}
  The uniformity of the norms with respect to the conductivity $\kappa$ can be established for every $\kappa\in [0,A]$ for every $A>0.$   Nevertheless,  as we shall see below, we need to use a  different transformation of the equation when
   $\kappa$ is close to one.
     \end{Rema}
  \begin{Rema}
  The result of the theorem remains true if we change  the Besov space  $B_{3,1}^0$ for 
   the Lebesgue space  $L^m$  for  $m>3.$ We have not tried to  get the best result in 
    terms of the regularity of the velocity.   At the price of more technicality, it is probably possible
      as in \cite{a} to get  the same result by assuming only that  $v_{0} \in {H}^{1 \over 2 }$.
      \end{Rema}
\begin{Rema}
By using the control of  $\nabla v$ in $B_{\infty, 1}^0$ which is given by Theorem  \ref{thm1},
 on can easily propagate  by classical arguments higher order regularity for example higher $H^s$ Sobolev regularity.
\end{Rema}

Let us explain  our strategy for the proof.  The crucial part in the proof consists now in finding
 suitable a priori estimates for  $(\zeta,\, \rho)$. Let us describe the idea 
  in the case that $\kappa=0$ (this  is the  one that one needs to understand in priority 
   because of the lack of smoothing effect). As we have already noticed the  coupling 
 between the two equations does not make the original
 Boussinesq system \eqref{bsintro}  well-suited for  a priori estimates.
  Since the right  hand side of  \eqref{mBouss} behaves roughly as  a Laplacian,
   we need to fully  use the left hand side to control it. 
    The main idea is to use an approach   related to the one  used  for  the study of   two-dimensional systems
     with a critical dissipation, see  \cite{HKR1,HKR2}. It consists in diagonalizing
      the linear part of the system  satisfied  by $\zeta$
       and  $\rho$.
      We introduce a   new unknown $\Gamma$ which formally reads 
     $$  \Gamma= \zeta - \big( \Delta + {2 \over r} \partial_{r}\big)^{-1}  {\partial_{r} \over r} \rho: = \zeta- \mathcal{L}
      \rho$$
      and we  study the system satisfied by $(\Gamma, \rho)$ which is given by:
      $$ \partial_{t} \Gamma + v \cdot \nabla \Gamma - \big(  \Delta + {2 \over r } \partial_{r}\big)
       \Gamma =  -\big[\mathcal{L}, v \cdot \nabla \rho\big], \quad  \partial_{t} \rho + v \cdot \nabla \rho =0$$
       where $\big[\mathcal{L}, v \cdot \nabla \big]$ is the commutator defined by 
       $$  [\mathcal{L}, v \cdot \nabla ]f= \mathcal{L}(v\cdot \nabla f) -  v \cdot \nabla \mathcal{L}f.$$
     We can thus  get a priori estimates for $\Gamma$
      and $\rho$ (note that they are  obvious if we negect the commutator)
       and then use them to  deduce estimates for $\zeta$.  The main difficulties that one has to deal with are twofold. The first one is the study of the operator $\mathcal{L}$, to make this argument rigorous, 
         we need to prove that $\mathcal{L}$ is well-defined and is a bounded operator
          on  $L^p$. The second one is the study of the commutator.
      Once estimates on $\zeta$ are obtained,  the situation gets close
       to the standard axisymmetric Navier-Stokes equations.
        When $\kappa$ is non zero,  it turns out that the same kind of transformation can be used  and
         that it depends smoothly on $\kappa$. The new unknown
          $\Gamma_{\kappa}= (1- \kappa) \zeta  - \mathcal{L}\rho$   solves
           the  same convection diffusion  equation as $\Gamma$. This is due to a surprising
            commutation property between $\Delta$ and $\mathcal{L}$
             which is stated in Lemma \ref{lemLD}. Consequently, we are again able to obtain
              and estimate for $\zeta$ from an estimate for $\Gamma_\kappa$ as soon as $\kappa \neq 1$.
          The case that $\kappa$ is close to one (which is easier since there is a non vanishing smoothing
           effect in the density equation) can be handled by using a different transformation.
       
     The paper is organized as follows. 
   In section \ref{sectionprelim} we  fix the notations,  give the definitions of the functional spaces
    that we shall use and  state some useful inequalities. Next, in section \ref{sectionel}, we
     study the operator $\mathcal{L}$, this amounts to study an elliptic equation with singular
      coefficients.  In  section \ref{sectionapriori}, we  obtain a priori estimates for
       sufficiently smooth solutions of \eqref{bsintro}, they are obtained by using the procedure
        that we have just described. Finally, in section \ref{sectionproof}, we give the proof
         of Theorem \ref{thm1}: we obtain the existence part by using the a apriori estimates
          and an approximation argument and then we prove the uniqueness part.

\section{Preliminaries}
\label{sectionprelim}
  Throughout this paper, $C$ stands for some real positive constant which may be different
  in each occurrence and $C_0$ for a positive constant depending on the initial data. 
   Moreover, both are assumed to be independent of $\kappa$ for $\kappa \geq 0$
    in a bounded set.
    We shall
  sometimes alternatively use the notation $X\lesssim Y$ for an inequality of the  type $X\leq CY$.

 For $s\in \RR,$  we  denote by $H^s(\mathbb{R}^3)$ the standard Sobolev spaces:   $u$ belongs to $H^s$ if $u$ is a tempered distribution and 
$$
 \|u\|_{ H^s}^2=\int_{\RR^3}(1+|\xi|^2)^{s}|\widehat u(\xi)|^2d\xi<\infty.
$$
 We shall also use the homogeneous version $\dot{H}^s(\mathbb{R}^3)$: for $s<3/2$, $u \in  \dot{H}^s$
  if $u  \in L^1_{loc}$ and
 $$  \|u\|_{ \dot{H}^s}^2=\int_{\RR^3}|\xi|^{2s}|\widehat u(\xi)|^2d\xi<\infty.$$
Now to introduce Besov spaces which are a generalization of Sobolev spaces we need to recall 
  the dyadic decomposition of the whole space (see  \cite{che1}).
  \begin{prop}
There exist two positive radial  functions  $\chi\in \mathcal{D}(\mathbb R^3)$ and
$\varphi\in\mathcal{D}(\mathbb R^3\backslash{\{0\}})$ such that
\begin{enumerate}
\item
$\displaystyle{\chi(\xi)
+\sum_{q\in\NN}\varphi(2^{-q}\xi)=1},\quad \frac13\le \chi^2(\xi)+\sum_{q\in\NN}\varphi^2(2^{-q}\xi)\le 1$\;
$\forall\xi\in\mathbb R^3,$
\item
$ \textnormal{supp }\varphi(2^{-p}\cdot)\cap
\textnormal{supp }\varphi(2^{-q}\cdot)=\varnothing,$ if  $|p-q|\geq 2$,\\

\item
$\displaystyle{q\geq1\Rightarrow \textnormal{supp}\chi\cap
\textnormal{supp }\varphi(2^{-q})=\varnothing}$.
\end{enumerate}
\end{prop}
For every $u\in{\mathcal S}'(\mathbb R^3)$ we define the nonhomogeneous Littlewood-Paley operators by,
$$
\Delta_{-1}u=\chi(\hbox{D})u;\, \forall
q\in\mathbb N,\;\Delta_qu=\varphi(2^{-q}\hbox{D})u\; \quad\hbox{and}\quad
S_qu=\sum_{-1\leq j\leq q-1}\Delta_{j}u.
$$
One can easily prove that for every tempered distribution $u,$ we have 
\begin{equation}\label{dr2}
u=\sum_{q\geq -1}\Delta_q\,u.
\end{equation}

In the sequel we will  frequently use Bernstein inequalities (see for \mbox{example \cite{che1}}).
\begin{lem}\label{lb}\;
 There exists a constant $C$ such that for $k\in\NN$, \mbox{$1\leq a\leq b$}   and $u\in L^a$, we have
\begin{eqnarray*}
\sup_{|\alpha|=k}\|\partial ^{\alpha}S_{q}u\|_{L^b}
&\leq&
C^k\,2^{q(k+3(\frac{1}{a}-\frac{1}{b}))}\|S_{q}u\|_{L^a},
\end{eqnarray*}
and for $q\in\NN$
\begin{eqnarray*}
\ C^{-k}2^
{qk}\|{\Delta}_{q}u\|_{L^a}
&\leq&
\sup_{|\alpha|=k}\|\partial ^{\alpha}{\Delta}_{q}u\|_{L^a}
\leq
C^k2^{qk}\|{\Delta}_{q} u\|_{L^a}.
\end{eqnarray*}
\end{lem}


 Let $(p,r)\in[1,+\infty]^2$ and $s\in\mathbb R,$ then the Besov
\mbox{space $B_{p,r}^s$} is
the set of tempered distributions $u$ such that
$$
\|u\|_{B_{p,r}^s}:=\Big( 2^{qs}
\|\Delta_q u\|_{L^{p}}\Big)_{\ell^{r}}<+\infty.
$$
We remark that  the  Sobolev space $H^s$ agrees with  the  Besov space  $B_{2,2}^s$.
Also,  a straightforward consequence of  the Bernstein inequalities  is the
 following continous embedding: 
\begin{equation}
\label{besemb}
B^s_{p_1,r_1}\hookrightarrow
B^{s+3({1\over p_2}-{1\over p_1})}_{p_2,r_2}, \qquad p_1\leq p_2\quad \mbox{ and}  \quad  r_1\leq r_2.
\end{equation}

\



%

 For any  Banach space $X$ with norm $\| \cdot \|_{X}$ and fonctions $f(t,x)$
  such that  for every $t$,  $f(t, \cdot)\in X$,   we shall use the notation $ \| f\|_{L^p_tX}
   =  \| \| f  \|_{X} \|_{L^p([0, T])}.$
   
   A useful application of Besov spaces is the following  logarithmic estimate
    for convection diffusion equations.
\begin{prop}
\label{log}
There exists $C>0$ such that  for every  $\kappa\geq 0$, $ p\in[1,\infty]$ and for
 every   $\rho$  \mbox{solution of}
$$
(\partial_t+v\cdot\nabla-\kappa\Delta)\rho=f, \quad \rho(0,x)= \rho_{0}(x)
$$
with $v$ a divergence free vector field, 
 the following estimate holds true
$$
\|\rho (t) \|_{B_{p,1}^0}\leq C\Big( \|\rho_0\|_{B_{p,1}^0}+\|f\|_{L^1_tB_{p,1}^0}\Big)\Big( 1+\int_0^t\|\nabla v(\tau)\|_{L^\infty}d\tau \Big), \quad \forall t \geq 0.
$$
\end{prop}
We refer to \cite{hk2} for the proof. Note that the amplification factor is only linear in 
$ \|\nabla v \|_{L^\infty}$.

\section{About an elliptic problem}
\label{sectionel}
The aim of this section is the study of the operator $\mathcal{L}= (\Delta+\frac2r)^{-1}\frac{\partial_r}{r} $.
This is the heart of the paper 
 since this is crucial to make rigorous the argument sketched in the introduction.
  This amounts to study  the regularity of the solution of an elliptic equation with singular coefficients.
   This is the goal  of the following proposition. 
\begin{prop}\label{prop1}
Let $\rho\in H^2(\RR^3)$    axisymmetric,    then
 there exists a unique axisymmetric  solution $f \in  H^2$ of the elliptic  problem
\begin{equation}
\label{eqel}
\Big(\Delta+\frac{2}{r}\partial_r\Big) f=\frac{\partial_r\rho}{r}.
\end{equation}
 Moreover,  for every $p\in [2, + \infty)$,  there exists an absolute constant $C_{p}>0$ such that:
\begin{equation}
\label{Lpprop1}
\|f\|_{L^p}\leq C_p\|\rho\|_{L^p}.
\end{equation}
\end{prop}
The important
  fact in this proposition is the $L^p$ estimate \eqref{Lpprop1} which only involves the $L^p$  norm  of $\rho$.
 An immediate consequence of Proposition \ref{prop1} is that $\mathcal{L}$  defines a bounded operator
  on $L^p$ for every $p \in [2, + \infty)$. 
    The  additional $H^2$ regularity   is  used to give a meaning to the equation.
      Because of the $1/r$ singularity on the axis, 
       we cannot    give a  meaning to the term $\partial_{r}f/r$ as a distribution when $f$ is merely $L^p$.
       When $f$ is in $H^2$, there is no problem  we have that $\partial_{r}f/r \in L^{1}_{loc}$ since for every compact set $K\subset\RR^3$
     $$  \|{\partial_{r} f  \over  r} \|_{L^1( K  )}
      \leq \| \nabla f \|_{L^6}  \, \|{1 \over r} \|_{L^{6 \over 5} (K)} <+ \infty$$
       thanks to the Sobolev embedding $H^1\subset L^6$ in dimension $3$.

\begin{proof}

 Let us first prove the existence of  a solution satisfying the required properties.
  We can  first assume that $\rho\in \mathcal{C}^\infty_{c}(\mathbb{R}^3)$ and
   then conclude by density.
 Since the  elliptic operator has some singular coefficients, we shall use an approximation
  argument. Since we have  by definition that $ r \partial_{r} = x_{h} \cdot \nabla$  with the notation
   $x_{h}= (x_{1}, x_{2}, 0), $ 
    we shall  consider for $\varepsilon>0$  the elliptic  problem
    \begin{equation}\label{eqel11}  \big(\Delta + {2 \over r^2 + \varepsilon} x_{h} \cdot \nabla \big) f= { 1 \over r^2 + \varepsilon}
     x_{h} \cdot \nabla \rho.
     \end{equation}
     Since the coefficients are not singular any more,  there is a 
       unique solution  $f^\eps$ for this problem given by the classical methods. By standard regularity arguments,
        this solution is in the Schwartz class
      and hence  the following a priori estimates are  justified.  Moreover, since $\rho$ is  axisymmetric, 
       $ f^\eps$ is also axisymmetric.      

 We shall first  prove that  the solution $f^\eps$ of \eqref{eqel11} satisfies   the estimate \eqref{Lpprop1}
  with a constant independent of $\eps$. In the proof of the a priori estimate, we shall  denote
   $f^\eps$ by $f$ for notational convenience.
  By taking the $L^2(\mathbb{R}^3)$ scalar  product of \eqref{eqel11} with    $(r^2+ \eps) |f|^{p-1}\hbox{sign}(f)$, we find  
 \begin{eqnarray}
 \label{id1}
&&\int_{\RR^3} \Delta f \,  |f|^{p-1}\hbox{sign}(f)(r^2+\eps)\,dx + 2 \int_{\RR^3} (x_{h} \cdot \nabla f ) \, |f|^{p-1}\hbox{sign}(f) \,dx\\
\nonumber&=&\int_{\RR^3} x_{h} \cdot \nabla \rho \,  |f|^{p-1}\hbox{sign}(f) \,dx.
\end{eqnarray}
 For the  first  term   in  the left-hand side, an integration by parts yields
\begin{eqnarray*}
 \int_{\RR^3} \Delta f \,  |f|^{p-1}\hbox{sign}(f)\, (r^2 + \eps) dx 
& =  &   - (p-1 )  \int_{\mathbb{R}^3} | f |^{p-2}\,  | \nabla f |^2\,(r^2+ \eps)  dx  \\
 & & -  2 
 \int_{\mathbb{R}^3} |f |^{p-1} \hbox{sign}(f) \, x_{h} \cdot  \nabla  f \, dx.  
 \end{eqnarray*}
  Consequently, we get that 
  \begin{eqnarray}
& & 
\label{id11}
 \int_{\RR^3} \Delta f \,  |f|^{p-1}\hbox{sign}(f)\,dx + 2 \int_{\RR^3} (x_{h} \cdot \nabla f ) \, |f|^{p-1}\hbox{sign}(f) \,dx
  \\
  \nonumber
  & &=  - (p-1 )  \int_{\mathbb{R}^3} | f |^{p-2}\,  | \nabla f |^2\,(r^2+ \eps)  dx.
  \end{eqnarray}
For the right-hand side of \eqref{id1}, we also obtain from an integration by parts that
\begin{eqnarray*}
\int_{\RR^3} x_{h} \cdot \nabla \rho \,  |f|^{p-1}\hbox{sign}(f) \,dx&   =  & - 2  \int_{\mathbb{R}^3}  \rho  |f |^{p-1}  \hbox{sign}(f) \, dx \\
& &  -  (p-1) \int_{\mathbb{R}^3}  \rho \, x_{h} \cdot \nabla f \,    \, |f|^{p-2} \, dx.
\end{eqnarray*}
 By using the H\"{o}lder inequality, this yields
 $$ 
  \Big| \int_{\RR^3} x_{h} \cdot \nabla \rho \,  |f|^{p-1}\hbox{sign}(f) \,dx \Big|
   \leq  2\,  \| \rho \|_{L^p} \,  \|f \|_{L^p}^{p-1} + (p-1)  \|f \|_{L^p }^{p-2 \over 2 }\,   \|  \rho \|_{L^p }
    \, \Big| \int_{\mathbb{R}^3}  |\nabla_{h} f|^2 \, |f |^{p-2}\, r^2 \, dx \Big|^{1\over 2}$$
    where $\nabla_{h}=(\partial_{1}, \partial_{2})^t$.
    By using  this last estimate, \eqref{id1} and \eqref{id11}, we  obtain that
     for every $\eps>0$
 \begin{eqnarray*}
 (p-1 )  \int_{\mathbb{R}^3} | f |^{p-2}\,  | \nabla f |^2\, r^2\,  dx & \leq
  &2\,  \| \rho \|_{L^p} \,  \|f \|_{L^p}^{p-1}   \\
   & &+ (p-1)  \|f \|_{L^p }^{p-2 \over 2 }\,   \|  \rho \|_{L^p }
    \, \Big| \int_{\mathbb{R}^3}  |\nabla_{h} f|^2 \, |f |^{p-2}\, r^2 \, dx \Big|^{1\over 2}.
    \end{eqnarray*}
  Consequently, by using the Young inequality, we get that 
\begin{equation}
\label{id13}
   \int_{\mathbb{R}^3} | f |^{p-2}\,  | \nabla f |^2\,r^2\,  dx  \leq
   C \Big(   \| \rho \|_{L^p} \,  \|f \|_{L^p}^{p-1} +  \|\rho\|_{L^p}^2 \, \|f \|_{L^p}^{p-2} \Big)
   \end{equation}
    for some $C>0$ independent of  $\eps$.
   To conclude,  we can  use the following inequality:   for $p\in[2,\infty[$, we have 
\begin{equation}\label{ineq}
\|f\|_{L^p}^p\le \frac{p^2}{4}\int|  \nabla_h f|^2 |f|^{p-2}\, r^2 dx.
\end{equation}
  This  is a special case of Caffarelli-Kohn-Nirenberg inequality \cite{CKN} (we shall recall the proof  
   will be given in the end of the proof of the proposition).  From \eqref{id13} and \eqref{ineq}, we obain that
  $$ \|f\|_{L^p}^2  \leq 
  C \Big(   \| \rho \|_{L^p} \,  \|f \|_{L^p}  +  \|\rho\|_{L^p}^2 \Big)$$
   where $C$ is  independent of $\eps$ and 
  and thus we obtain the estimate \eqref{Lpprop1}
   for the solution of \eqref{eqel11} by using the Young inequality.

  When  $\varepsilon$ goes to zero,  we get  the existence of    $f\in L^p $ and a subsequence $ \eps_{n}$
  such that   the solution $f^{\eps_{n}}$ of  \eqref{eqel11}  converges weakly to $f$
    which  satisfies the desired  estimate
$$
\|f\|_{L^p}\le C_p\|\rho\|_{L^p}.
$$
    In order to get that  $f$ solves the equation \eqref{eqel}, we need  more information on 
     $f^\eps$. Indeed, the difficulty is to give a meaning to the term $\partial_{r}f/r$
      which is not well-defined as a distribution when $f$ is merely $L^p$.
             We shall use a  uniform $H^2$  estimate for the solution of \eqref{eqel11}
        but which also  involves  the $H^2$ norm  of $\rho$. This is why we have required more regularity on 
        $\rho$. 
       At first, let us notice that since we assume that $\rho$ is  in $L^2$, we have
        that $f^\eps$ is uniformly bounded in $L^2$. Next,   
        we multiply \eqref{eqel11} by $\Delta f$, we get that 
  \begin{equation}
  \label{energyprop1} \int_{\mathbb{R}^3} | \Delta  f|^2 \, dx + 2 
     \int_{\mathbb{R}^3}   {x_{h}  \cdot \nabla f \over r^2 + \eps} \Delta f  \, dx=
       \int_{\mathbb{R}^3} {x_{h} \cdot \nabla \rho \over r^2 + \eps} \, \Delta f  \, dx.
      \end{equation}
    The right-hand side can be estimated by
  \begin{eqnarray*}
    \Big| \int_{\mathbb{R}^3} {x_{h} \cdot \nabla \rho \over r^2 + \eps} \,\Delta f \, dx \Big|
  &   \leq &   \|  {  \partial_{r} \rho  \over r } \|_{L^2} \|\Delta f \|_{L^2}.
   \end{eqnarray*}
  Since we assume that $\rho$ is axisymmetric, we can use  that
  $$   {\partial_{r} \rho \over r } =  \Delta \rho - \partial^2_{zz}-  \partial_{r}^2 \rho=
  {x_{2}^2 \over r^2} \partial_{1}^2 \rho + {x_{1}^2 \over r^2} \partial_{2}^2\rho - 2{x_{1} x_{2} \over r^2}
   \partial_{12}^2\rho$$
    and hence that we have the estimate
   $$  \big\|  {  \partial_{r} \rho  \over r } \big\|_{L^2} \leq  4 \| \rho \|_{H^2}.$$
  This yields
  \begin{equation}
  \label{energy2prop1}
   \Big| \int_{\mathbb{R}^3} {x_{h} \cdot \nabla \rho \over r^2 + \eps} \,\Delta f \, dx \Big|
     \leq   4  \| \rho  \|_{H^2} \|\Delta f \|_{L^2}.
  \end{equation}
    
  Next, we can study the second term in the left-hand side of \eqref{energyprop1}.
   We have by integration by parts that
   \begin{eqnarray*}
   2   \int_{\mathbb{R}^3}   {x_{h}  \cdot \nabla f \over r^2 + \eps} \Delta f  \, dx
  &  =  & - \int_{\mathbb{R}^3}  {x_{h} \over r^2 + \eps } \cdot \nabla \big( |\nabla f |^2 \big)
   -  2 \int_{\mathbb{R}^3} \Big( \nabla  \big({ x_{h} \over r^2 + \eps} \big) \cdot \nabla \Big) f \cdot \nabla f \,  \\
    & = &    \int_{\mathbb{R}^3} \nabla \cdot \big(   {x_{h} \over r^2 + \eps } \big)\, |\nabla f |^2
     -   2  \int_{\mathbb{R}^3} \Big( \nabla  \big({ x_{h} \over r^2 + \eps} \big) \cdot \nabla \Big) f \cdot \nabla f
    dx.
   \end{eqnarray*}
Next,  we  infer  
  $$  \nabla \cdot \big( {x_{h} \over r^2 + \eps} \big)=  { 2 \eps \over (r^2 + \eps)^2}$$
  and since $f$ is axisymmetric
  $$ \Big( \nabla  \big({ x_{h} \over r^2 + \eps} \big) \cdot \nabla \Big) f \cdot \nabla f
  = \Big({1 \over r^2 + \eps} - {2 r^2  \over (r^2+ \eps)^2} \Big) |\partial_{r} f|^2.$$
  This yields 
\begin{equation}
\label{energy3prop1}   2   \int_{\mathbb{R}^3}   {x_{h}  \cdot \nabla f \over r^2 + \eps} \Delta f  \, dx
 = 2 \int_{\mathbb{R}^3} { r^2 \over (r^2 + \eps)^2} |\partial_{r} f |^2 \, dx \geq 0.
 \end{equation}
  Consequently, we get from \eqref{energyprop1}, \eqref{energy2prop1} and \eqref{energy3prop1}
   that 
   $$ \|\Delta f \|_{L^2}  \leq   4 \|\rho \|_{H^2}$$
    and hence that 
    \begin{equation}
    \label{fepsH1}
     \| f \|_{H^2} \leq C \|\rho \|_{H^2}
     \end{equation}
     with $C$ independent of $\eps$.
      From this uniform $H^2$ estimate for $f^\eps$, we   get  that $f^{\eps_{n}}$
       (up to a subsequence not relabelled)
       converges weakly in $H^2$  to some $f \in H^2$ and then that 
        $f$   is a weak solution of   \eqref{eqel}.
         This ends the proof of the existence of  a solution. The uniqueness  is  a consequence  
           of  the  standard energy inequality.

     Let us now come back to the proof of \eqref{ineq}. 
Since  $\hbox{div} (x_h)=2$, by  integrating by parts, we obtain
\begin{eqnarray*}
2\|f\|_{L^p}^p&=&-\int_{\RR^3}x_h\cdot\nabla_h(|f|^p)dx\\
&=&-p\int_{\RR^3}(r\partial_r f)|f|^{p-1}\hbox{sign}(f)dx\\
&\le& p\Big(\int_{\RR^3}|r\partial_r f|^2 |f|^{p-2}dx\Big)^{\frac12}\|f\|_{L^p}^{\frac{p}{2}}.
\end{eqnarray*}
Therefore we find the desired estimate,
$$
\|f\|_{L^p}^p\le\frac{p^2}{4}\int_{\RR^3}|r\partial_r f|^2 |f|^{p-2}dx.
$$

 This ends the proof of Proposition \ref{prop1}.

  \end{proof}

  In the proof of the main result, we shall also need to use  the operator
   $\big( \Delta + {2 \over r} \partial_{r}\big)^{-1}\partial_{z}/r$.
     The aim of the following proposition is to define rigorously this operator.

 \begin{prop}\label{prop11}
Let $\rho\in L^2(\RR^3)$  be axisymmetric such that $\partial_{z} \rho /r \in L^2(\mathbb{R}^3)$,     then
 there exists a unique  axisymmetric  solution $f \in  H^2$ of the elliptic  problem
\begin{equation}
\label{eqel1}
\Big(\Delta+\frac{2}{r}\partial_r\Big) f=\frac{\partial_z\rho}{r}.
\end{equation}
 Moreover,  for every $p\in [2, + \infty)$,  there exists an absolute constant $C_{p}>0$ such that:
\begin{equation}
\label{Lpprop11}
\|f\|_{L^p}\leq C_p\|\rho\|_{L^p}.
\end{equation}
\end{prop}

 Again, the important fact is the estimate \eqref{Lpprop11} 
  which only involves the $L^p$ norm of  $\rho$. From this estimate, we get   that the operator
  $\big( \Delta + {2 \over r} \partial_{r}\big)^{-1}\partial_{z}/r$  is a bounded operator on $L^p$.
  The additional regularity $\rho\in L^2$, $\partial_{z} \rho/r\in L^2$ is again only used to get 
     the $H^2$ regularity  on the solution which allows to give a meaning to  the equation.
     
  Note that the assumptions on $\rho$  here are different from the one of Proposition \ref{prop1}.
   This comes from the fact that we shall need to  use 
    $\big( \Delta + {2 \over r} \partial_{r} \big)^{-1} {\partial_{r} \over r} \rho$
     when $\rho$ is a smooth solution of \eqref{bsintro} whereas, we shall only need
      to use $ \big( \Delta + {2 \over r} \partial_{r} \big)^{-1} {\partial_{z} \over r} \big(v^r \rho\big)$
       where $(v, \rho)$ is a smooth solution of \eqref{bsintro}.
               In the first case, as we have seen in the proof, 
         the    regularity on $\rho$ ensures that $\partial_{r}\rho/r$ is in $L^2$.  In a similar way, 
         in  the second case
           for a smooth solution of \eqref{bsintro} such that $\omega/r
\in L^2$, we indeed have that  ${\partial_{z } \over r} ( v^r \rho)\in L^2$.

 \begin{proof}
 The proof follows the same lines as the proof  of Proposition \ref{prop1}.
  Consequently, we shall just indicate the main difference.
   We now consider the regularized problem
  \begin{equation}
  \label{reg2} \big(\Delta + {2 \over r^2 + \varepsilon} x_{h} \cdot \nabla_{h} \big) f= {\partial_{z} \rho \over \sqrt{r^2 + \eps}}. \end{equation}
   By  multiplying the equation with $(r^2 + \eps) |f|^{p-1} \textnormal{sign}(f)$, we get by using again \eqref{id11}
   and an integration by parts for the right-hand side that 
 $$  \int_{\mathbb{R}^3} |f|^{p-2} \, |\nabla f |^2(r^2 + \eps)\, dx
 \leq  \int_{\mathbb{R}^3} \sqrt{r^2 + \eps} \, |  \rho| \, |f|^{p-2}\,  |\partial_{z} f | \, dx.$$
 From the H\"{o}lder inequality, we obtain that 
$$ \Big( \int_{\mathbb{R}^3} |f|^{p-2} \, |\nabla f |^2 r^2\, dx  \Big)^{1\over 2} 
 \leq \|\rho \|_{L^p } \, \|f \|_{L^p }^{p- 2 \over 2}$$
  and hence by using again \eqref{ineq}, we finally get  that
  $$ \|f \|_{L^p} \leq C_{p} \|\rho \|_{L^p}.$$
 This will give the estimate \eqref{Lpprop11}  by passing to the limit.
 
 To prove  an  $H^2$ estimate on $f$, we again multiply  \eqref{reg2} by $\Delta f$, 
  since $f$ is axisymmetric, we get 
  from \eqref{energy3prop1}  that
 $$ \| \Delta f \|_{L^2} \leq  \|{\partial_{z} \rho \over r} \|_{L^2}$$
 and this provides the $H^2$ estimate for $f$. We can then pass to the limit to get a solution of 
  \eqref{eqel1} with the  claimed properties. The uniqueness follows from the classical energy estimate.
  This ends the proof of Proposition \ref{prop11}.
  
 \end{proof}

The aim of the next two lemmas is to prove some  identities involving
 the operator $\mathcal{L}=\big(\Delta+{2\over r} \partial_{r }\big)^{-1}\frac{\partial_r}{r}$ which will be useful  to get  the equation satisfied by $\mathcal{L}\rho$
  in order to diagonalize the system.
  
  At first, we have
\begin{lem}\label{leme1} For any  smooth axisymmetric  function $f$ we have
 the identity : 
 $$ \mathcal{L} \,\partial_{r} f = {f\over r} - \mathcal{L}( {f \over r })  - \partial_{z}  \big( \Delta + {2 \over r } \partial_{r}
  \big)^{-1}  {\partial_{z} f \over r }.$$

\end{lem}
\begin{proof}
 We first obtain that 
\begin{eqnarray*}
\frac1r\partial_{rr}f&=&\partial_r(\frac1r\partial_rf)+\frac{1}{r^2}\partial_rf\\
&=&\partial_r\big(\partial_r( {f \over r})+{f\over r^2}\big)+\frac1{r^2}\partial_rf\\
&=&\partial_{rr}({f \over r})+{2\over r^2} \partial_r f -{2 \over r^3}f\\
&=&\big(\partial_{rr}+ {2 \over r }\partial_r\big)({f \over r})\\
&=&\big(\Delta+{  2\over r}\partial_{r}\big)({f \over r})-{1 \over r }\partial_r({f \over r})-\partial_{zz}({f\over r }).
\end{eqnarray*}
It follows that
\begin{eqnarray*}
\big(\Delta+{2\over r} \partial_{r }\big)^{-1}(\frac1r\partial_{rr}\, f)&=&{f \over r}-\big(\Delta+ {2\over r} \partial_{r}\big)^{-1}(\frac1r\partial_{r}(
{f \over r }))-\big(\Delta+{2 \over r} \partial_{r }\big)^{-1}(\partial_{zz}({f \over r }))
\\
&=&{f \over r}-\mathcal{L}({f \over  r})-\partial_z\big(\Delta+{2\over r} \partial_{r }\big)^{-1}({\partial_{z} f \over r}).
\end{eqnarray*}
Note that  we  have used   the fact that the operators $\big(\Delta+{2 \over  r} \partial_{r}\big)^{-1}$ and $\partial_z$ commute since the coefficients of the operator $\Delta+\frac{2}{r}\partial_r$ do not depend on the variable $z$.
 This is the desired identity and therefore, this ends the proof of Lemma \ref{leme1}.
\end{proof}

We shall also use the following:
\begin{lem}
\label{lemLD}
For every smooth axisymmetric function $\rho$, we have the identity:
$$  \mathcal{L}\, \Delta \rho= \big( \Delta + {2 \over r} \partial_{r} \big) \, \mathcal{L} \rho.$$
\end{lem}

\begin{proof}
 At first, by direct computations, we find that
 \begin{equation}
 \label{com1} \big[ \Delta, {1 \over r  } \partial_{r} \big]  = - 2 \big( { 1 \over r^2} \partial_{r}^2 - {1\over r^3  }\partial_{r}\big)
  =  - { 2 \over r} \partial_{r}\big( {1 \over r } \partial_{r} \cdot \big).\end{equation}
 Now, let us consider $f = \mathcal{L} \, \rho$. By definition, $f$ solves the elliptic equation
 $$ \big( \Delta + { 2 \over r } \partial_{r} \big)  f =  { 1 \over r } \partial_{r} \rho.$$
 Consequently, we get that 
 $  u =  \big( \Delta + { 2 \over r } \partial_{r} \big)  f $ solves the elliptic equation:
$$  \big( \Delta + { 2 \over r } \partial_{r} \big) u=   \big( \Delta + { 2 \over r } \partial_{r} \big)\big(  { 1 \over r } 
 \partial_{r} \rho \big).$$
 By using the formula \eqref{com1}, we  obtain for    the right-hand side 
 $$    \big( \Delta + { 2 \over r } \partial_{r} \big)\big(  { 1 \over r } 
 \partial_{r} \rho \big) = { 1 \over r} \partial_{r} \big( \Delta \rho  +  { 2 \over r } \partial_{r} \rho  \big) + 
   \big[ \Delta, {1\over r } \partial_{r} \big] \rho=
   { 1 \over r} \partial_{r}    \big( \Delta \rho  +  { 2 \over r } \partial_{r} \rho - {2 \over r } \partial_{r} \rho  \big)
   = {1 \over r } \partial_{r} \Delta \rho.$$
 This proves that $u$ solves the equation
 $$\big( \Delta + { 2 \over r } \partial_{r} \big) u=   {1 \over r } \partial_{r} \Delta \rho$$
  and hence that $u = \mathcal{L} \Delta \rho$.
   Since $u=   \big( \Delta + { 2 \over r } \partial_{r} \big)  f= 
     \big( \Delta + { 2 \over r } \partial_{r} \big) \mathcal{L}\, \rho$, this ends the proof
      of \mbox{Lemma \ref{lemLD}.}

\end{proof}

\section{A priori estimates}
\label{sectionapriori}
This section is devoted to the a priori estimates needed for the proof of Theorem \ref{thm1}. We shall prove
two   results: the first one deals with some basic  energy estimates.  The second one
 which is  more difficult  
deals  with the control of some stronger norms.

\begin{prop}\label{Energy}
Let $(v,\rho)$  be a smooth solution of \eqref{bsintro} then
\begin{enumerate}
\item for $p\in[1,\infty]$ and $t\in\RR_+$,  we have
$$
\|\rho(t)\|_{L^p}\le\|\rho_0\|_{L^p}, 
$$
\item for $v_0\in L^2, \rho_0\in L^2$ and $t\in\RR_+$ we have
$$
\|v(t)\|_{L^2}^2+\int_0^t\|\nabla v(\tau)\|_{L^2}^2d\tau\le C_0(1+t^2),
$$
where $C_0$ depends only on $\|v_0\|_{L^2}$ and $\|\rho_0\|_{L^2}$.

\end{enumerate}

\end{prop}
Note that  the axisymmetric assumption is not needed in this proposition
\begin{proof}
The first estimate is classical for  convection diffusion  equations with a divergence free
 vector field. For  $p=\infty$ it is just   the maximum principle, while for finite $p$, it is
  a consequence of the  fact that 
  $$\big( \partial_{t }+ v \cdot \nabla  - \kappa \Delta \big) |\rho|^p \leq 0.$$
For the second one we take the $L^2(\mathbb{R}^3)$ scalar  product of  the velocity equation with $v$.
 From  integration by parts and the  fact that $v$ is divergence free, we get
\begin{equation}\label{eqs1}
\frac12\frac{d}{dt}\|v(t)\|_{L^2}^2+\|\nabla v(t)\|_{L^2}^2\le\|v(t)\|_{L^2}\|\rho(t)\|_{L^2}.
\end{equation}
 This yields 
$$
\frac{d}{dt}\|v(t)\|_{L^2}\le\|\rho(t)\|_{L^2}.
$$
 By  integration  in time, we find that 
$$
\|v(t)\|_{L^2}\le\|v_0\|_{L^2}+\int_0^t\|\rho(\tau)\|_{L^2}d\tau.
$$
Since  $\|\rho(t)\|_{L^2} \leq \|\rho_0\|_{L^2},$  we infer
$$
\|v(t)\|_{L^2}\le\|v_0\|_{L^2}+t\|\rho_0\|_{L^2}.
$$
Plugging this estimate into \eqref{eqs1} gives
$$
\frac12\|v(t)\|_{L^2}^2+\int_0^t\|\nabla v(\tau)\|_{L^2}^2d\tau
\le
\frac12\|v_0\|_{L^2}^2+\big(\|v_0\|_{L^2}+t\|\rho_0\|_{L^2}\big)\|\rho_0\|_{L^2} t.
$$
This gives the second claimed estimate and ends  the proof of  the proposition.
\end{proof}

We shall next prove the following result.
\begin{prop}\label{Strong}
Let $v_0\in H^1,$ with $\omega_0/r\in L^2$ and $\rho_0\in L^{2}\cap L^3$. Then any smooth solution $(v,\rho)$ of  \eqref{bsintro} with $\rho$ axisymmetric and $v$ axisymmetric without swirl satisfies:
\begin{enumerate}
\item for every $t\in\RR_+$
$$
\Big\| {\omega \over r} (t)\Big\|_{L^2}\le C_0 e^{C_0 t}, 
$$
\item for every $t\in\RR_+$
$$
\|v(t)\|_{H^1}^2+\int_0^t\|v(\tau)\|_{H^2}^2d\tau\le C_{0} e^{C_0 t},
$$
where $C_0$ depends only  on the norms of the initial data.
\end{enumerate}
\end{prop}
Note that the axi-symmetry is crucial in this proposition. The estimates  are uniform for $\kappa \geq0$
 in a bounded  set.
\begin{proof}
 We shall  use the notation $\zeta= \omega_{\theta}/r$ where the vorticity $\omega$ is given
  by  $\omega= \omega_{\theta} e_{\theta}$ since the flow is axisymmetric. 
 The equation for $\zeta$ reads
\begin{equation}
\label{equation_i}
\big(\partial_t+v\cdot\nabla\big)\zeta-\big(\Delta+{{2 \over r}}\partial_r\big) \zeta =-\frac{\partial_r\rho}{r}\cdot
\end{equation}
 By using the operator $\mathcal{L}$,   this equation can be written as 
\begin{equation}
\label{eqm}
\big(\partial_t+v\cdot\nabla\big)\zeta -\big(\Delta+{2\over r}\partial_r\big)\big( \zeta -\mathcal{L}\rho\big)=0.
\end{equation}
Applying $\mathcal{L}$ to the equation  for  $\rho$ we first  get
$$
\partial_t\mathcal{L}\rho+v\cdot\nabla\mathcal{L}\rho - \kappa  \mathcal{L}  \Delta \rho=-[\mathcal{L}, v\cdot\nabla]\rho
$$
 and hence by using Lemma \ref{lemLD}, we find
 \begin{equation}
 \label{Lrho}
\partial_t\mathcal{L}\rho+v\cdot\nabla\mathcal{L}\rho - \kappa\big(  \Delta+ {2 \over r}\partial_{r}\big) \mathcal{L}\rho=-[\mathcal{L}, v\cdot\nabla]\rho.
 \end{equation}
 In view of  \eqref{eqm} and \eqref{Lrho}, we can set
  $$\Gamma:=  (1- \kappa )\zeta  -\mathcal{L}\rho.$$ 
   We find that $\Gamma$ solves the equation
$$
\big(\partial_t+v\cdot\nabla\big)\Gamma-\big(\Delta+{2 \over r}\partial_r\big)\Gamma=\textnormal{div }(v\,\mathcal{L}\rho)-\mathcal{L}(v\cdot\nabla\rho).
$$
Taking the $L^2(\mathbb{R}^3)$ inner product with $\Gamma$ and integrating by parts
 in the usual way  we get since $v$ is divergence free that 
\begin{eqnarray}
\label{idI}
\frac12\frac{d}{dt}\|\Gamma(t)\|_{L^2}+\|\nabla\Gamma(t)\|_{L^2}^2&\le& \|\nabla \Gamma(t)\|_{L^2}\|v\mathcal{L}\rho\|_{L^2}\\
\nonumber &-&\int_{\RR^3}\mathcal{L}(v\cdot\nabla\rho)\Gamma dx\\
\nonumber&=&\hbox{I}+\hbox{II}.
\end{eqnarray}
To estimate the first term we use successively  the  H\"older inequality,  Proposition \ref{prop1} and
 the first estimate of Proposition \ref{Energy} to get 
\begin{eqnarray*}
\|v\,\mathcal{L}\rho\|_{L^2}&\le&\|v\|_{L^6}\|\mathcal{L}\rho\|_{L^3}
\lesssim\|v\|_{L^6}\|\rho\|_{L^3} 
 \lesssim   \|v\|_{L^6}\|\rho_{0}\|_{L^3}
\end{eqnarray*}
Now from  the Young inequality, we  get 
\begin{equation}
\label{estI}
\textnormal{I}\le\frac14\|\nabla\Gamma\|_{L^2}^2+C\|v\|_{L^6}^2\|\rho_0\|_{L^3}^2.
\end{equation}
Towards the estimate of the  second term in the right-hand side of \eqref{idI},  we  can use 
 since $v$ is divergence free that 
$$
v\cdot\nabla\rho=\partial_r(v^r\rho)+{v^r \over r}\rho+\partial_z(v^z\rho).
$$
Thus we obtain
\begin{eqnarray*}
\mathcal{L}(v\cdot\nabla\rho)&=&\mathcal{L}\partial_{r}(v^r\rho)+\mathcal{L}({v^r\over r}\rho)\\
&+&\partial_z\mathcal{L}(v^z\rho).
\end{eqnarray*}
 Next, we can use Lemma \ref{leme1} to get 
\begin{eqnarray*}
\mathcal{L}(v\cdot\nabla\rho)&=&{v^r \over r} \rho-\partial_z\big(\Delta+{2\over r} \partial_{r }\big)^{-1}(\frac1r\partial_{z}(v^r\rho))
\\
&+&\partial_z\mathcal{L}(v^z\rho).
\end{eqnarray*}
Thus we find that 
\begin{eqnarray*}
\hbox{II}&=&-\int {v^r \over r} \rho\,\Gamma dx+\int\big(\Delta+{2 \over r } \partial_{r }\big)^{-1}(\frac1r\partial_{z}(v^r\rho))\partial_z\Gamma dx
\\
&-&\int\mathcal{L}(v^z\rho)\partial_z\Gamma dx\\
&=&J_1+J_2+J_3.
\end{eqnarray*}
 From Cauchy-Schwarz inequality, we first obtain that
 $$ |J_{2}+ J_{3}| \leq  \|\partial_z\Gamma\|_{L^2}\Big(\Big\| (\Delta + {2 \over r } \partial_{r}\big)^{-1}
  \big( { \partial_{z} \over r } \big)(v^r\rho)\Big\|_{L^2} +\| \mathcal{L}(v^z\rho)\|_{L^2} \Big)
   \leq  \|\nabla\Gamma\|_{L^2} \big(\|v^r\rho\|_{L^2} +\|v^z\rho\|_{L^2} \big)$$
    where the last estimate comes from Proposition \ref{prop11} and Proposition \ref{prop1}.
    
     Next, 
     by 
using   successively the H\"older inequality,    the Sobolev inequality
\begin{equation}
\label{sob3}
 \| f \|_{L^6} \lesssim \| \nabla f \|_{L^2}
 \end{equation}
 in dimension $3$, 
    and the first estimate of Proposition \ref{Energy},  we infer
\begin{eqnarray*}
|J_2+J_3|
&\le&\|\nabla\Gamma\|_{L^2}\|v\|_{L^6}\|\rho\|_{L^3}\\
&
\le& \|\nabla\Gamma\|_{L^2}\|\nabla v\|_{L^2}\|\rho_0\|_{L^3}.
\end{eqnarray*}
 In a similar way, we estimate $J_{1}$:
\begin{eqnarray*}
|J_1|&\le&\|\Gamma\rho\|_{L^{\frac65}}\|v^r/r\|_{L^6}\\
&\le& 
\|\Gamma\|_{L^2}\|\rho\|_{L^{3}}\|v^r/r\|_{L^6}\\
&\le&
\| \Gamma\|_{L^2}\|\rho_0\|_{L^{3}}\|v^r/r\|_{L^6}.
\end{eqnarray*}
 To estimate $\|v^r/r\|_{L^6}$, we can   use the    inequality 
$$
|v^r/r|\lesssim \frac{1}{|\cdot|^2}\star(|\omega_\theta/r|)
$$
which is very useful  in the field of axisymmetric solution  of  incompressible fluid mechanics equations, 
 we refer to  \cite{Taira,rd}  for the proof. Next, 
since $\frac{1}{|x|^2}\in L^{\frac32,\infty}(\mathbb{R}^3)$, we get
 from the  classical Hardy-Littlewood-Sobolev inequality that \begin{equation}
 \label{hls}
\|v^r/r\|_{L^6}\lesssim \|\omega_\theta/r\|_{L^2} = \|\zeta \|_{L^2}.\end{equation}
 From the definition of $\Gamma$ we have that 
 \begin{equation}
 \label{mGamma}
  \|\zeta \|_{L^2} \leq |\kappa-1|^{-1}\Big(\| \Gamma \|_{L^2} + \| \mathcal{L} \rho \|_{L^2}\Big).\end{equation}
  Note that this estimate is uniform for $ \kappa \geq 0$ and far from $\kappa=1$. We will see later how to obtain uniform estimates around the value $\kappa=1.$
   Therefore, by using again Proposition \ref{prop1} and the first estimate of Proposition \ref{Energy}, we
    obtain 
    \begin{eqnarray*}
\|v^r/r\|_{L^6}     
\lesssim  \|\Gamma\|_{L^2}+\|\mathcal{L}\rho\|_{L^2}
\lesssim \|\Gamma\|_{L^2}+\|\rho\|_{L^2}
\lesssim \|\Gamma\|_{L^2}+\|\rho_0\|_{L^2}.
\end{eqnarray*}
Consequently
$$
|J_1|\lesssim\|\Gamma\|_{L^2}\|\rho_0\|_{L^{3}}\big(\|\Gamma\|_{L^2}+\|\rho_0\|_{L^2}\big).
$$
Combining theses estimates with the  Young inequality we get
\begin{eqnarray}
\nonumber |\hbox{II}|&\lesssim&\|\nabla\Gamma\|_{L^2}\|\nabla v\|_{L^2}\|\rho_0\|_{L^{3}}+\|\Gamma\|_{L^2}\|\rho_0\|_{L^{3}}\big(\|\Gamma\|_{L^2}+\|\rho_0\|_{L^2}\big)\\
\label{estII}&\le&\frac14\|\nabla\Gamma\|_{L^2}^2+C\|\nabla v\|_{L^2}^2\|\rho_0\|_{L^3}^2+C\|\rho_0\|_{L^{3}}\|\Gamma\|_{L^2}^2+C\|\Gamma\|_{L^2}\|\rho_0\|_{L^3}\|\rho_0\|_{L^2}.
\end{eqnarray}
It follows from  \eqref{idI}, \eqref{estI} and \eqref{estII} that
\begin{eqnarray*}
\frac{d}{dt}\|\Gamma(t)\|_{L^2}^2+\|\nabla\Gamma(t)\|_{L^2}^2&\lesssim&\|\nabla v\|_{L^2}^2\|\rho_0\|_{L^3}^2 +\|\rho_0\|_{L^{3}}\|\Gamma\|_{L^2}^2+\|\Gamma\|_{L^2}\|\rho_0\|_{L^3}\|\rho_0\|_{L^2}\\
&\lesssim&\|\rho_0\|_{L^{3}}\|\Gamma\|_{L^2}^2+\|\nabla v\|_{L^2}^2\|\rho_0\|_{L^3}^2+\|\rho_0\|_{L^3}\|\rho_0\|_{L^2}^2.
\end{eqnarray*}
 We can then integrate  in time and use the energy inequality  of Proposition \ref{Energy}-(2)  to get 
\begin{eqnarray*}
\|\Gamma(t)\|_{L^2}^2+\int_0^t\|\nabla\Gamma(\tau)\|_{L^2}^2d\tau&\lesssim&\|\rho_0\|_{L^3}^2\|\nabla v\|_{L^2_tL^2}^2+\|\rho_0\|_{L^{3}}\|\rho_0\|_{L^2}^2 t\\
&+&\|\rho_0\|_{L^{3}}\int_0^t\|\Gamma(\tau)\|_{L^2}^2d\tau\\
&\le&C_0(1+t^2)+C_{0}\int_0^t\|\Gamma(\tau)\|_{L^2}^2d\tau.
\end{eqnarray*}
 By using  the Gronwall inequality we find
$$
\|\Gamma(t)\|_{L^2}^2+\int_0^t\|\nabla\Gamma(\tau)\|_{L^2}^2d\tau\le C_0e^{C_0 t}.
$$
It follows from a new use of \eqref{mGamma} that 
\begin{equation}\label{eq00}
\|\zeta(t)\|_{L^2}\le C_0 e^{C_0 t}.
\end{equation}
This proves (1) in Proposition \ref{Strong}.

To prove (2), we can now perform an energy estimate on the equation \eqref{tourbillon}
 satisfied by $\omega_{\theta}$. 
 By taking  the $L^2(\mathbb{R}^3)$ scalar product of \eqref{tourbillon} with $\omega_\theta$  
  we get 
   \begin{eqnarray*}
\frac12\frac{d}{dt}\|\omega_\theta(t)\|_{L^2}^2+\|\nabla\omega_\theta\|_{L^2}^2+\|\omega_\theta/r\|_{L^2}^2&\le&\int_{\RR^3}v^r(\omega_\theta/r)\omega_\theta dx-\int_{\RR^3}\partial_r\rho\omega_\theta dx.
\end{eqnarray*}
 Thanks to an integration by parts, we have that 
$$ \Big|\int_{\RR^3}\partial_r\rho\omega_\theta dx \Big|  \leq  \|\rho \|_{L^2}\, \big( \|\nabla \omega_{\theta} \|_{L^2}+
 \| \omega_{\theta}/r \|_{L^2}\big)
$$
 and hence,  by using the Holder inequality and the Sobolev inequality \eqref{sob3}, we find
 \begin{eqnarray*}
& & \frac12\frac{d}{dt}\|\omega_\theta(t)\|_{L^2}^2+\|\nabla\omega_\theta\|_{L^2}^2+\|\omega_\theta/r\|_{L^2}^2 \\
&\le&\|v\|_{L^3}\|\omega_\theta/r\|_{L^2}\|\omega_\theta\|_{L^6}+\|\rho\|_{L^2}\big(\|\nabla\omega_\theta\|_{L^2}
 + \|\omega_{\theta}/r \|_{L^2}\big)\\
&\lesssim&\|v\|_{L^3}\|\omega_\theta/r\|_{L^2}  \|\nabla \omega_{\theta}\|_{L^2}+\|\rho\|_{L^2} \big(\|\nabla \omega_\theta\|_{L^2}
 + \|\omega_{\theta}/r \|_{L^2}\big).
\end{eqnarray*}
Thus  from  \eqref{eq00},  the first estimate of Proposition \ref{Energy} and  the Young 
inequality, we obtain
\begin{equation}
\label{fin1}
\frac{d}{dt}\|\omega_\theta(t)\|_{L^2}^2+ {1 \over 2} \Big( \|\nabla\omega_\theta\|_{L^2}^2+\|\omega_\theta/r\|_{L^2}^2 \Big) \lesssim\|\rho_0\|_{L^2}^2+C_0 e^{C_0 t}\|v(t)\|_{L^3}^2.
\end{equation}
By interpolation, the  Sobolev embedding \eqref{sob3} and  the second estimate of Proposition
 \ref{Energy}  we  have 
\begin{eqnarray*}
\int_{0}^t \|v\|_{L^3}^2&\le& \int_{0}^t \|v\|_{L^2}\|v\|_{L^6}
\lesssim \int_{0}^t  \|v\|_{L^2}\|\nabla v\|_{L^2}
\leq C_0(1+t^2).
\end{eqnarray*}
Therefore we get by integrating \eqref{fin1} in time that 
$$
\|\omega_\theta(t)\|_{L^2}^2+\int_0^t\big(\|\nabla\omega_\theta(\tau)\|_{L^2}^2+\|\frac{\omega_\theta}{r}(\tau)\|_{L^2}^2\big)d\tau \le e^{C_0 t}.
$$
 Since we have $\| \omega \|_{L^2}= \| \omega_{\theta} \|_{L^2}$ and 
$$\|\nabla\omega\|_{L^2}^2=\|\nabla\omega_\theta\|_{L^2}^2+\|\omega_\theta/r\|_{L^2}^2, 
$$
 we finally obtain that  
$$
\|\omega(t)\|_{L^2}^2+\int_0^t\|\nabla\omega(\tau)\|_{L^2}^2d\tau \le e^{C_0 t}.
$$
This ends the proof of Proposition \ref{Strong} where $\kappa$ belongs to a compact set that does not contain $1$. Let us now see how to get uniform bounds around $\kappa=1$ which is more easy and does not require the use of the operator $\mathcal{L}.$  We write the equation of the density under the form
$$
\partial_t\rho+v\cdot\nabla\rho-(\Delta+\frac{2}{r}\partial_r)\rho=(1-\kappa)\Delta\rho-\frac{2}{r}\partial_r\rho.
$$
We set $\Gamma_1=\zeta-\frac\rho2.$ Then combining this equation with  \eqref{equation_i} we find
$$
\partial_t\Gamma_1+v\cdot\nabla\Gamma_1-(\Delta+\frac{2}{r}\partial_r)\Gamma_1=\frac{\kappa-1}{2}\Delta\rho.
$$
Taking the $L^2$ scalar product of this equation with $\Gamma_1$  and integrating by parts we get
\begin{eqnarray*}
\frac12\frac{d}{dt}\|\Gamma_1(t)\|_{L^2}^2+\|\nabla\Gamma_1(t)\|_{L^2}^2&\le& \frac{k-1}{2}\|\nabla\rho\|_{L^2}\|\nabla\Gamma_1\|_{L^2}\\
&\le&\frac12(\frac{\kappa-1}{2})^2\|\nabla\rho\|_{L^2}^2+\frac12\|\nabla\Gamma_1\|_{L^2}^2.
\end{eqnarray*}
Integrating in time yields
$$
\|\Gamma_1(t)\|_{L^2}^2+\|\nabla\Gamma_1\|_{L^2_tL^2}^2\le C{(\kappa-1)^2}\|\nabla\rho\|_{L^2_tL^2}^2.
$$
Combining this estimate with the energy estimate
$$
\kappa^{\frac12}\|\nabla\rho\|_{L^2_tL^2}\le\|\rho_0\|_{L^2}
$$
gives
$$
\|\Gamma_1(t)\|_{L^2}^2+\|\nabla\Gamma_1\|_{L^2_tL^2}^2\le C\frac{(\kappa-1)^2}{\kappa^{\frac12}}\|\rho_0\|_{L^2}^2.
$$
This gives the desired result.
\end{proof}
The following proposition gives some  more precise information than stated in  Theorem \ref{thm1} about the solution. This will be useful to prove the  uniqueness result.
\begin{prop}\label{Lipschitz}
Let $v_0\in H^1$ be a divergence free  axisymmetric without swirl vector field such that  $\omega_0/r\in L^2$ and $\rho_0\in L^{2}\cap L^m, m>3$ or $\rho_0\in L^{2}\cap B_{3,1}^0$ an axisymmetric  function. Then any smooth  solution  $(v, \rho)$ of the system \eqref{bsintro}
   satisfies for every $p\in]3,\infty]$
$$
\|v\|_{L^1_tB_{p,1}^{1+\frac3p}}+\|\nabla v\|_{L^1_tL^\infty}\le C_0e^{C_0 t}.
$$
The estimate is uniform with respect to  $\kappa$ lying in a bounded set.
\end{prop}
\begin{proof}
We first prove the result in the  case of   $\rho_0\in L^{2}\cap B_{3,1}^0.$ 
Let $q\in\NN$ and set $v_q:=\Delta_q v.$ Then applying the operator $\Delta_q$
 from the Littlewood-Paley decomposition to the velocity equation and using Duhamel formula we get
\begin{equation}
\label{fin2}
v_q(t)=e^{t\Delta}v_q(0)+\int_0^te^{(t-\tau)\Delta}\Delta_q\mathcal{P}(v\cdot\nabla v)(\tau,x)d\tau+\int_0^te^{(t-\tau)\Delta}\Delta_q\mathcal{P}(\rho e_z)(\tau,x)d\tau,
\end{equation}
where $\mathcal{P}$ is the Leray projection on divergence free vector fields.
Now we will use  two estimates: the first one is proved in \cite{che1}
$$
\|e^{t\Delta}\Delta_q f\|_{L^p}\le Ce^{-ct2^{2q}}\|\Delta_q f\|_{L^p}.
$$
The second estimate is
$$
\|\Delta_q\mathcal{P}f\|_{L^p}\le C\|\Delta_q f\|_{L^p}.
$$
This last estimate is a consequence of the fact that $\Delta_q\mathcal{P}=\psi(2^{-q}\textnormal{D})$ with $\psi\in \mathcal{D}(\RR^3)$.

Therefore, we get from \eqref{fin2} that 
$$
\|v_q\|_{L^1_tL^p}\lesssim 2^{-2q}\|v_q(0)\|_{L^p}+2^{-2q}\int_0^t\|\Delta_q(v\cdot\nabla v)(\tau)\|_{L^p}d\tau+2^{-2q}\|\Delta_q\rho\|_{L^1_t L^p}.
$$
It follows from the above inequality,  the Besov embeddings \eqref{besemb} and Proposition \ref{Energy} that 
\begin{eqnarray*}
\|v\|_{L^1_tB_{p,1}^{1+\frac3p}}&\le&\|\Delta_{-1}v\|_{L^1_tL^p}+ \|v_0\|_{B_{p,1}^{-1+\frac3p}}+\|v\cdot\nabla v\|_{L^1_tB_{p,1}^{-1+\frac3p}}+\|\rho\|_{L^1_tB_{p,1}^{-1+\frac3p}}\\
&\le& Ct\|v\|_{L^\infty_t L^2}+\|v_0\|_{H^1}+\|v\otimes v\|_{L^1_tB_{2,1}^{\frac32}}+\|\rho\|_{L^1_tB_{3,1}^{0}}\\
&\le& C_0(1+t^2)+\|v\otimes v\|_{L^1_tB_{2,1}^{\frac32}}+\|\rho\|_{L^1_tB_{3,1}^{0}}.
\end{eqnarray*}
Since $B_{2,1}^{\frac32}$ is an algebra, we have
$$
\|v\otimes v\|_{L^1_tB_{2,1}^{\frac32}}\le C\|v\|_{L^2_tB_{2,1}^{\frac32}}^2.
$$Moreover,  the embedding $H^{2}\hookrightarrow B_{2,1}^{\frac32}$ combined with
 the second estimate of   Proposition \ref{Strong} gives
$$
\|v(t)\|_{L^2_tB_{2,1}^{\frac32}}\le C_0e^{C_0t}.
$$
Consequently, we obtain
\begin{equation}\label{line}
\|v\|_{L^1_tB_{p,1}^{1+\frac3p}}\le C_0 e^{C_0 t}+\|\rho\|_{L^1_tB_{3,1}^{0}}.
\end{equation}
It remains to estimate the norm of the  density. For this purpose,  we first  use  the logarithmic estimate described in Proposition \ref{log}  and \eqref{besemb} to get 
\begin{eqnarray*}
\|\rho(t)\|_{B_{3,1}^0}&\le&C\|\rho_0\|_{B_{3,1}^0}\Big(1+\int_0^t\|v(\tau)\|_{B_{\infty,1}^1}d\tau\Big)\\
&\le&C\|\rho_0\|_{B_{3,1}^0}\Big(1+\int_0^t\|v(\tau)\|_{B_{p,1}^{1+\frac3p}}d\tau\Big).
\end{eqnarray*}
Set $V(t):=\|v\|_{L^1_tB_{p,1}^{\frac3p}}$, then combining this estimate with \eqref{line} yields
\begin{eqnarray*}
V(t)&\le& C_0 e^{C_0 t}+\|\rho_0\|_{B_{3,1}^0}\int_0^tV(\tau)d\tau.
\end{eqnarray*}
We conclude now by Gronwall lemma. 

Let us now come back to the case that  $\rho_0\in L^{2}\cap L^m,$ with $m>3$ which is more easy than the previous case.  
 The same proof as above  for the velocity yields 
$$ \|v\|_{L^1_tB_{p,1}^{1+\frac3p}}\le C_0 e^{C_0 t}+\|\rho\|_{L^1_tB_{p,1}^{ -{1 + {3 \over p } } } }$$
 and it still remains to estimate the density.  
 Let us  set $m_1:=\min(m,p)>3$ then by Besov embeddings  and the first estimate of
  Proposition \ref{Energy}, we get
\begin{eqnarray*}
\|\rho\|_{L^1_t B_{p,1}^{-1+\frac3p}}\lesssim \|\rho\|_{L^1_t B_{m_1,1}^{-1+\frac3{m_1}}}
\lesssim  \|\rho\|_{L^1_t  L^{m_1} } \lesssim  t \|\rho_{0} \|_{L^{m_{1}}}
\lesssim  t\|\rho_0\|_{L^{2}\cap L^m}.
\end{eqnarray*}
Note that the last estimate holds by interpolation.
 This ends the proof.
\end{proof}
\section{Proof of the main result}
\label{sectionproof}
For the existence part of Theorem
\ref{thm1} we smooth out the initial data as follows
$$
v_{0,n}=S_n v_0, \rho_{0,n}=S_n \rho_0,
$$
where $S_n$ is the cut-off in frequency defined in section 2. 
We  start with  the following stability results.
\begin{lem}\label{lem1}
Let $v_0$ be a free divergence axisymmetric vector-field without swirl and $\rho_0$ an axisymmetric scalar function. Then
\begin{enumerate}
\item for every $n\in\NN$, $v_{0,n}$ and $\rho_{0,n}$ are axisymmetric and $\textnormal{div }v_{0,n}=0.$

\item
If $v_0\in H^1$  is  such that
$({\textnormal{curl }v_0})/{r}\in L^2$
 and $\rho_{0}\in L^2 \cap B_{3, 1}^0$. Then  there exists a constant $C$ independent of $n$ such that
$$
\|v_{0,n}\|_{H^1}\le \|v_0\|_{H^1},\quad
\big\| (\textnormal{curl }v_{0,n})/r \big\|_{L^2}\le C\big\| {(\textnormal{curl }v_0)}/{r} \big\|_{L^2}, 
 $$
 $$  \|\rho_{0, n}\|_{L^2} \leq \|\rho_{0}\|_{L^2}, \quad
  \|\rho_{0, n} \|_{B_{3, 1}^0} \leq C \|\rho_{0} \|_{B_{3, 1}^0}.
$$

\end{enumerate}
\end{lem}
\begin{proof}
We have $v_{0,n}=2^{3n}\chi(2^n\cdot)\star v_0$. The  fact that the vector field $v_{0,n}$ is axisymmetric is due to the radial property  of the functions
$\chi$, for more details see \cite{A-H-K}. The estimate of $v_{0,n}$ in  $H^1$ is
easy to obtain by using the classical properties of the convolution  laws.  The proof of  the second estimate
 for the velocity  is more subtle and  we refer to \cite{benamer} where it
is proven  in the   general framework of  Lebesgue spaces, that is in  $L^p$, \mbox{for all $p\in[1,\infty]$.}
 The estimates for the density follow by standard convolution inequalities.
\
\end{proof}
 We have just seen in  Lemma \ref{lem1} that  the initial
structure of axisymmetry is preserved for every $n$ and the involved
norms are uniformly controlled with respect to this parameter $n$.
Thus we can construct locally in time a unique solution
$(v_n,\rho_n).$ This solution is globally defined  since the
Lipschitz norm of the velocity does not blow up in finite  time as it was stated
in \mbox{Proposition \ref{Lipschitz}.}
 Note that  since  $\rho_{0}\in B_{3, 1}^0$ (or $L^m, \, m>3$), the $L^3$ norm
  of $\rho_{0, n}$ is also uniformly bounded  thanks to the  continuous embedding
   $ B_{3, 1}^0 \subset L^3 $.
 By standard arguments we can
show that this family $(v_n,\rho_n)$ converges to $(v,\rho)$ which satisfies in
turn our initial problem. We omit here the details  and we will next focus on the
uniqueness part. Set
$$
\mathcal{X}_T:=\big(L^\infty_TH^1\cap L^2_TH^2\cap L^1_TB_{p,1}^{1+\frac3p}\big)\times L^\infty_T H^{-1}, \quad\hbox{for some} \quad p\in[3,\infty[.
$$
Let $(v^i,\rho^i)\in \mathcal{X}_T, 1\leq i\leq 2$ be  two solutions of the
 system \eqref{bsintro} with the same initial \mbox{data $(v_0,\theta_0)$} and denote
 $\delta v=v^2-v^1,\delta\theta=\theta^2-\theta^1$. Then
\begin{equation}\label{diff}
\left\{ \begin{array}{ll}
\partial_t\delta v-\Delta\delta
v
=-\mathcal{P}(v^2\cdot\nabla\delta v)-\mathcal{P}(\delta v\cdot\nabla v^1)+\mathcal{P}(\delta\rho\, e_z)
\\
\partial_t \delta\rho +v^2\cdot\nabla\delta\rho - \kappa \Delta\,  \delta \rho=
 -\delta v\cdot\nabla \rho^1
\\
{\mathop{\rm div}}\,  v^i=0.
\end{array}
\right.
\end{equation}
Using the maximal smoothing of the heat operator combined with H\"older inequality we get
\begin{eqnarray}\label{diff00}
\nonumber\|\delta v\|_{L^\infty_t\dot H^{1}}&\lesssim& \|v^2\cdot\nabla \delta v\|_{L^2_tL^2}+\|\delta v\cdot\nabla v^1\|_{L^2_tL^2}+\|\delta \rho\|_{L^1_t\dot H^{-1}}\\
\nonumber&\lesssim&\Big(\int_0^t\|v^2(\tau)\|_{L^\infty}^2\|\delta v(\tau)\|_{\dot H^1}^2d\tau\Big)^{\frac12}+\Big(\int_0^t\|\delta v(\tau)\|_{\dot H^1}^2\| v^1(\tau)\|_{\dot H^{\frac32}}^2d\tau\Big)^{\frac12}\\
&&\qquad\qquad\qquad\qquad\qquad\qquad\qquad+\|\delta \rho\|_{L^1_t\dot H^{-1}}.
\end{eqnarray}
We have used the classical law product
\begin{eqnarray*}
\|fg\|_{L^2}\lesssim \|f\|_{\dot H^{1}}\|g\|_{\dot H^{\frac12}}.
\end{eqnarray*}
To estimate $\|\delta\rho\|_{H^{-1}}$ we will use 
Proposition 3.1 of  \cite{AP} : for every $p\in[2,\infty[$
\begin{eqnarray*}
\|\delta\rho(t)\|_{\dot H^{-1}}
&\le&
C\|\delta v\cdot\nabla\rho^1\|_{L^1_t\dot H^{-1}}
\exp{(C\|\nabla v^2\|_{L^1_tB^{\frac{3}{p}}_{p,1}})}.
\end{eqnarray*}
We remark that the proof of this result was done in the inviscid case but it can be extended to the viscous case with uniform bounds with respect to the parameter $\kappa.$

As ${\mathop{\rm div}}\, \delta v=0,$ then using H\"older inequality and Sobolev embedding $\dot H^1\hookrightarrow L^6$ yield
\begin{eqnarray*}
\|\delta v\cdot\nabla\rho^1\|_{L^1_t\dot H^{-1}}&\le&\|\delta v\,\rho^1\|_{L^1_tL^2}\\
&\le&\|\delta v\|_{L^1_tL^6}\|\rho^1\|_{L^\infty_tL^3}\\
&\le& \|\rho_0\|_{L^3}\|\delta v\|_{L^1_t\dot H^1}.
\end{eqnarray*}
Thus we get
\begin{eqnarray*}
\|\delta\rho(t)\|_{\dot H^{-1}}
&\le&
C\|\rho_0\|_{L^3}\exp{(C\|\nabla v^2\|_{L^1_tB^{\frac{3}{p}}_{p,1}})}\|\delta v\|_{L^1_t\dot H^1}.
\end{eqnarray*}
 By plugging  this estimate into
(\ref{diff00}), we  finally get
\begin{eqnarray}\nonumber
\nonumber\|\delta v\|_{L^\infty_t\dot H^{1}}^2&\lesssim&\Big(\int_0^t\|v^2(\tau)\|_{L^\infty}^2\|\delta v(\tau)\|_{\dot H^1}^2d\tau+\int_0^t\|\delta v(\tau)\|_{\dot H^1}^2\| v^1(\tau)\|_{\dot H^{\frac32}}^2d\tau\\
\nonumber &+&\|\rho_0\|_{L^3}^2\exp{(C\|\nabla v^2\|_{L^1_tB^{\frac{3}{p}}_{p,1}})}\,t\,\int_0^t\|\delta v(\tau)\|_{\dot H^{1}}^2d\tau.
\end{eqnarray}
Since $v^2\in L^2_t L^\infty$ and $v^1\in L^2_t\dot H^{\frac32}$ then we get the uniqueness by using  Gronwall inequality.

\end{document}